\documentclass[11pt]{amsart}

\usepackage{amssymb,amsmath,url,graphicx, amscd}
\usepackage{hyperref}
\usepackage{stmaryrd} 
\usepackage{caption}
\usepackage{enumerate}
\usepackage{subcaption}
\usepackage[all]{xy}
\usepackage{tikz,tikz-3dplot}
\usepackage[normalem]{ulem}

\theoremstyle{plain}
\newtheorem{theorem}{Theorem}
\newtheorem{lemma}[theorem]{Lemma}

\theoremstyle{definition}
\newtheorem{definition}{Definition}
\newtheorem{proposition}{Proposition}
\newtheorem{corollary}{Corollary}
\newtheorem{example}{Example}

\theoremstyle{remark}
\newtheorem{remark}{Remark}

\begin{document}

\title{Convexity and Thimm's Trick}

\author{Jeremy Lane}
\address{Dept.\ of Mathematics, University of Toronto, 40 St.\ George Street,
Toronto Ontario M5S 2E4, Canada}
\email{jeremy.lane@mail.utoronto.ca}

\date{\today}



\maketitle

\begin{abstract}
	In this paper we study topological properties of maps constructed by Thimm's trick with Guillemin and Sternberg's action coordinates on a connected Hamiltonian $G$-manifold $M$. Since these maps only generate a Hamiltonian torus action on an open dense subset of $M$, convexity and fibre-connectedness of such maps do not follow immediately from Atiyah-Guillemin-Sternberg's convexity theorem, even if $M$ is compact.  The core contribution of this paper is to provide a simple argument circumventing this difficulty.
			
	In the case where the map is constructed from a chain of subalgebras we prove that the image is given by a list of inequalities that can be computed explicitly in many examples. This generalizes the fact that the images of the classical  Gelfand-Zeitlin systems on coadjoint orbits are Gelfand-Zeitlin polytopes.  Moreover, we prove that if such a map generates a completely integrable torus action on an open dense subset of $M$, then all its fibres are smooth embedded submanifolds.
\end{abstract}

\section{Introduction}

A connected symplectic manifold $(M,\omega)$ equipped with a Hamiltonian action of a compact torus $T$ generated by a momentum map $\mu\colon M \rightarrow \mathfrak{t}^*$ is a \emph{proper Hamiltonian $T$-manifold}\footnote{
This differs slightly from the definition of proper Hamiltonian $T$-manifolds given in \cite{karshon-tolman}, which requires that the convex subset of $\mathfrak{t}^*$ is also open. This extra assumption is not necessary in the context of this paper.} 
if $\mu$ is proper as a map to a convex subset of $\mathfrak{t}^*$. The convexity theorem for proper Hamiltonian torus manifolds says that if $(M,\omega,\mu)$ is a proper Hamiltonian $T$-manifold, then $\mu(M)$ is convex, the fibres of $\mu$ are connected, and $\mu$ is open as a map to its image (cf. \cite[Theorem 30]{kb} or \cite{bor1}). If in addition the action of $T$ on $M$ is effective and completely integrable, then  $(M,\omega,\mu)$ is called a \emph{proper toric $T$-manifold}. It follows from \cite[Theorem 1.3]{kl} that proper toric $T$-manifolds are classified up to isomorphism by $\mu(M)$ together with the weight lattice of the torus.  If $M$ is compact, then this classification reduces to Delzant's theorem \cite{delzant}. 

Let $G$ be a compact connected Lie group. It was observed by Guillemin and Sternberg that a collective integrable system constructed by Thimm's trick on a Hamiltonian $G$-manifold $(M,\omega,\Phi)$ admits natural action coordinates on an open dense subset $\mathcal{U} \subseteq M$ \cite{gs1}. The most important examples  of this construction are the classical Gelfand-Zeitlin\footnote{ One should note that there are multiple spellings of Zeitlin in the literature, including Tsetlin and Cetlin. Works by Kostant-Wallach, Kogan-Miller, and Guillemin-Sternberg respectively each use a different spelling.} systems on $U(n)$ and $SO(n)$ coadjoint orbits for which Guillemin and Sternberg's action coordinates generate effective completely integrable Hamiltonian torus actions on open dense subsets\footnote{This is proven for regular coadjoint orbits in \cite{gs1}. A proof for arbitrary coadjoint orbits can be found in \cite{pabiniak}.}. The images of the classical Gelfand-Zeitlin systems were shown in \cite{gs1} to be convex polytopes defined by the interlacing inequalities for eigenvalues of Hermitian matrices, which are called \emph{Gelfand-Zeitlin polytopes}.

In \cite[Proposition 3.1]{pabiniak1}, it is claimed that for $U(n)$ and $SO(n)$ coadjoint orbits the open dense subsets where the classical Gelfand-Zeitlin systems define Hamiltonian torus actions are 
proper toric manifolds\footnote{The definition of proper Hamiltonian $T$-manifold given in \cite{pabiniak1} is that of \cite{karshon-tolman}, but the assumption of openness for the convex subset of $\mathfrak{t}^*$ is actually not needed in the context of \cite{pabiniak1}. } and this claim is combined with the classification of proper toric manifolds to deduce tight lower bounds for $U(n)$ and $SO(n)$ coadjoint orbits' Gromov width from the geometry of their Gelfand-Zeitlin polytopes.  Unfortunately, a review of the literature cited in \cite{pabiniak1} and related work does not yield a direct explanation of why these open dense subsets are proper Hamiltonian torus manifolds. 

In this paper we describe the general construction of a map $F$ by ``Thimm's trick with Guillemin and Sternberg's action coordinates'' (Equation \eqref{continuous map}). We observe that if $(M,\omega,\Phi)$ is a connected Hamiltonian $G$-manifold, then it follows from properties of Hamiltonian $G$-manifolds that the open dense subset $\mathcal{U}\subseteq M$ where $F$ generates a Hamiltonian action of a big torus $T'$ is connected (Lemma \ref{connected}). For a classical Gelfand-Zeitlin system on a $U(n)$ or $SO(n)$ coadjoint orbit, it then follows that the open dense subset where $F$ defines an effective, completely integrable torus action is a proper toric manifold (Example \ref{gz system}).  More generally, we prove the following theorem.

\begin{theorem}\label{main theorem} Suppose that $(M,\omega, \Phi)$ is a connected Hamiltonian $G$-manifold and $F$ is a map constructed on $M$ by Thimm's trick with Guillemin and Sternberg's action coordinates.  If $F$ is proper, then 
\begin{enumerate}
	\item $(\mathcal{U}, \omega\vert_{\mathcal{U}},F\vert_{\mathcal{U}})$ is a proper Hamiltonian $T'$-manifold (where $T'$ is the big torus mentioned above, see the discussion preceding Lemma \ref{connected} for details of the definition of $T'$),
	\item $F(M)$ is convex, and
	\item The fibres of $F$ are connected.
\end{enumerate}
\end{theorem}

We also prove a straightforward generalization of the fact that the images of the classical Gefland-Zeitlin systems are Gelfand-Zeitlin polytopes: if $(M,\omega,\Phi)$ is a connected Hamiltonian $G$-manifold with $\Phi$ proper, then the image of the map $F$ constructed by Thimm's trick with Guillemin and Sternberg's action coordinates from a chain of subalgebras 
$$ \mathfrak{h}_1 \subseteq \ldots \subseteq \mathfrak{h}_k = \mathfrak{g}$$
is the locally polyhedral set defined by the inequalities of the momentum set $\Phi(M) \cap \mathfrak{t}_+^*$ and the branching inequalities corresponding to the chain of subalgebras (Proposition \ref{branching system}). Moreover, if the torus action generated by $F$ on an open dense subset is completely integrable, then we prove that the fibres of $F$ are embedded submanifolds (Proposition \ref{smooth fibres}).

If $(M,\omega,\Phi)$ is a multiplicity free Hamiltonian $U(n)$ or $SO(n)$-manifold, then the map $F$ constructed by Thimm's trick with Guillemin and Sternberg's action coordinates from chains of subalgebras
\begin{equation}\label{gelfand-zeitlin chains}
	\mathfrak{u}(1) \subseteq \mathfrak{u}(2) \subseteq  \cdots \subseteq \mathfrak{u}(n) \mbox{ or } \mathfrak{so}(2) \subseteq \mathfrak{so}(3) \subseteq  \cdots \subseteq \mathfrak{so}(n)
\end{equation}
respectively generates a completely integrable Hamiltonian torus action on an open dense subset of $M$. If $\Phi$ is proper, then it follows by Theorem \ref{main theorem} that this open dense subset is a proper, completely integrable Hamiltonian torus manifold (in general, the torus action may not be effective) and by Proposition \ref{branching system} one has an explicit description of the image of $F$.  This can then be used to prove explicit lower bounds for the Gromov widths of a much larger family of symplectic manifolds than was studied in \cite{pabiniak1} (namely multiplicity free $U(n)$ and $SO(n)$-manifolds with proper momentum maps) \cite{lane1}. Unfortunately, applying Thimm's trick to similar chains of subalgebras for groups other than $U(n)$ and $SO(n)$ does not yield completely integrable torus actions \cite[p. 225]{gs3}, although one should note that Harada was able to extend the construction of classical Gelfand-Zeitlin systems to construct completely integrable systems on $Sp(n)$ coadjoint orbits in \cite{harada}.  

Nishinou-Nohara-Ueda proved that the classical Gelfand-Zeitlin systems on $U(n)$ coadjoint orbits can be constructed by toric degeneration \cite{nnu}.  This was later generalized by Harada-Kaveh who proved that, under various technical assumptions, a toric degeneration of a smooth projective variety endows an open dense subset with the structure of a proper toric manifold \cite[Theorem B]{harada-kaveh}.  These results  were applied in \cite{kaveh} to prove lower bounds for the Gromov width of smooth projective varieties in terms of their Newton-Okounkov bodies and in \cite{halacheva-pabiniak} and \cite{flp} to finish the proof of tight lower bounds for the Gromov width of coadjoint orbits of compact simple Lie groups (tight upper bounds were proven \cite{castro}).

In contrast to toric degeneration, the construction of Hamiltonian torus actions by Thimm's trick with Guillemin and Sternberg's action coordinates 
\begin{itemize}
	\item does not require the manifold to be projective or even K\"ahler and
	\item typically does not yield a completely integrable torus action.
\end{itemize}
Recently, Hilgert-Martens-Manon described a symplectic analogue of toric degeneration called ``symplectic contraction'' and showed that maps constructed via symplectic contraction are the same as the maps constructed by Thimm's trick \cite{hmm}. The relation between this paper and the results of \cite{hmm} is discussed in Section \ref{s:symplectic contraction}.

The contents of this paper are as follows. In Section 2 we recall basic definitions and results pertaining to Hamiltonian group actions. Section 3 studies the details of Thimm's trick with Guillemin and Sternberg's action coordinates and gives the proof of Theorem \ref{main theorem}. Section 3 also contains a brief subsection describing how maps constructed by Thimm's trick interact with symplectic reduction.  In Section 4 we recall the general branching inequalities, multiplicity free $G$-manifolds, and prove Proposition \ref{branching system} and Proposition \ref{smooth fibres}. In Section 5 we illustrate the Thimm construction and applications in symplectic topology with several low-dimensional examples.  In Section \ref{s:symplectic contraction} we discuss the related work by \cite{hmm}.

The author would like to thank Yael Karshon for suggesting the study of convexity for Gelfand-Zeitlin systems and our many discussions. The author  would also like to thank the referees for their helpful comments and noticing a gap in an earlier proof of Proposition \ref{thimm 1}. The author was supported by a NSERC PGSD scholarship during work on this paper.

\section{Hamiltonian Group Actions}\label{s:hamiltonian group actions}

\subsection{Basic Definitions}\label{s:basic definitions}

Let $G$ be a compact, connected Lie group with Lie algebra $\mathfrak{g}$. We write 
\begin{equation}
  \langle-,-\rangle:\mathfrak{g}^*\times \mathfrak{g} \rightarrow \mathbb{R}
\end{equation}
for the dual pairing, $Ad_{g}X$ for the adjoint action of $g \in G$ on $\mathfrak{g}$, and $Ad_{g}^*\xi$ for the coadjoint action of $G$ on $\mathfrak{g}^*$.
Given an action of $G$ on a manifold $M$, let $\underline{X}$ denote the fundamental vector field of $X\in \mathfrak{g}$. A manifold $M$ is \emph{symplectic} if it is equipped with a closed, non-degenerate 2-form $\omega$.  Recall,

\begin{definition} 
An action of $G$ on a symplectic manifold $(M,\omega)$ is \emph{Hamiltonian} if there is an equivariant map $\Phi: M \rightarrow \mathfrak{g}^*$
such that 
\begin{equation}\label{momentum map}
  \iota_{\underline{X}}\omega = d\langle \Phi,X\rangle
\end{equation}
for all $X \in \mathfrak{g}$. If this is the case, then $\Phi$ is called a \emph{momentum map} for the action and $(M, \omega,\Phi)$ is called a \emph{Hamiltonian} $G$\emph{-manifold}.
\end{definition}

Given a function $f\in C^{\infty}(M)$, the \emph{Hamiltonian vector field} $X_f$ is defined by Hamilton's equation
\begin{equation}\label{hamiltons equation}
  \iota_{X_f}\omega = df.
\end{equation}

\begin{definition} 
A \emph{Poisson bracket} on $C^{\infty}(M)$ is a Lie bracket $\{\cdot,\cdot\}$ such that $\{f,\cdot\}$ is a derivation for all $f\in C^{\infty}(M)$. A map between manifolds with Poisson brackets, $\Phi:(M,\{\cdot,\cdot\}) \rightarrow(M',\{\cdot,\cdot\}')$, is \emph{Poisson} if for all $f,g \in C^{\infty}(M')$,
\begin{equation}
  \{f\circ \Phi,g\circ \Phi \} = \{f,g\}'\circ \Phi.
\end{equation}
Two functions $f,g \in C^{\infty}(M)$ are said to \emph{Poisson commute} if $\{f,g\} = 0$.
\end{definition} 

The Poisson bracket on a symplectic manifold is
\begin{equation}
  \left\{ f,g \right\} = \omega\left(X_f, X_g\right).
\end{equation}
The Kostant-Kirillov-Souriou Poisson bracket on $\mathfrak{g}^*$, is defined by
\begin{equation}
  \left\{ f,g\right\}_{\xi} = \langle \xi , [df_{\xi},dg_{\xi}]\rangle
\end{equation}
where the linear functional $df_{\xi}:T_{\xi}\mathfrak{g}^* = \mathfrak{g}^*\rightarrow \mathbb{R}$ is identified with an element of $\mathfrak{g}$.  The symplectic leaves of this Poisson bracket are the coadjoint orbits which are equipped with the Kostant-Kirillov-Souriau symplectic form. If $\Phi: M \rightarrow \mathfrak{g}^*$ is a momentum map for a Hamiltonian $G$-action, then $\Phi$ is Poisson with respect to these brackets \cite{audin}.

\subsection{Properties of the Sweeping Map}\label{s:properties of the sweeping map}

In this section we recall important facts about the sweeping map. These appear in various places, such as \cite{gs5, cdm, hnp} along with a generalization to orbifolds in \cite{lmtw}.

Fix a choice of maximal torus $T \subseteq G$ with Lie algebra $\mathfrak{t}$. Let $\mathfrak{t}_+$ be a choice of closed\footnote{By this we mean that $\mathfrak{t}_+$ is the closed polyhedral cone in $\mathfrak{t}$ defined as an intersection of closed half-spaces corresponding to reflections generating of the Weyl group.} positive Weyl chamber. Since $G$ is compact there is a non-degenerate, positive definite, bilinear form $(-,-)$ on $\mathfrak{g}$ which is invariant under the adjoint action of $G$.  The map $X \mapsto (X,-)$ is a $G$-equivariant vector space isomorphism of $\mathfrak{g}$ with $\mathfrak{g}^*$, with respect to the adjoint and coadjoint actions.  We also call the image of $\mathfrak{t}_+$ under this map a \emph{positive Weyl chamber} and denote it by $\mathfrak{t}_+^*$. The positive Weyl chamber is a fundamental domain for the coadjoint action of $G$.

\begin{definition}\label{sweeping}
 Let $G$ be a compact, connected Lie group and let $\mathfrak{t}_+^*$ be a positive Weyl chamber. The \emph{sweeping map} $s:\mathfrak{g}^* \rightarrow \mathfrak{t}_{+}^*$ is defined by letting $s(\xi)$ be the unique element of the set $\left(G\cdot \xi\right) \cap \mathfrak{t}^*_+$. 
\end{definition}

The sweeping map is continuous and induces a homeomorphism $\mathfrak{g}^*/G \cong \mathfrak{t}_+^*$.  If $\sigma$ is a stratum\footnote{$\mathfrak{t}_+^*$ has a natural stratification as a polyhedral set. The maximal stratum is the interior of $\mathfrak{t}_+^*$ and the lower dimensional strata are the relative interiors of the intersections of the faces of $\mathfrak{t}_+^*$.} of the  polyhedral cone $\mathfrak{t}_+^*$, then $\Sigma_{\sigma} = G\cdot \sigma$ is a connected component of an orbit-type stratum in $\mathfrak{g}^*$ and the restricted map $s: \Sigma_{\sigma} \rightarrow \sigma$ is smooth. We recall the following detail of the symplectic cross-section theorem (see e.g. \cite[Theorem 3.1]{lmtw}).


  \begin{theorem}\label{principal stratum}
    Let $(M,\omega,\Phi)$ be a connected Hamiltonian $G$-manifold.  There exists a unique stratum $\sigma \subseteq \mathfrak{t}_{+}^*$ with the property that $\Phi(M) \cap \mathfrak{t}_{+}^* \subseteq \overline{\sigma}$ and $\Phi(M) \cap \sigma$ is non-empty.
  \end{theorem}
  The unique stratum $\sigma$ of the preceding theorem is called the \emph{principal stratum} of $M$. Unpacking the details of the proof of Theorem \ref{principal stratum}, one has the following proposition which will be useful in the proof of Theorem \ref{main theorem}.

  \begin{proposition}\label{preimage of open face}
     Let $\sigma$ be the principal stratum of a connected Hamiltonian $G$-manifold $(M,\omega,\Phi)$. The pre-image $\Phi^{-1}(\Sigma_{\sigma})$ is a connected, dense open submanifold of $M$ and its complement is contained in a locally finite union of submanifolds of codimension at least 2 in $M$.
  \end{proposition}

The stabilizer subgroup $G_{\xi}$ of a point $\xi$ in a stratum $\sigma$ of $\mathfrak{t}_+$ is independent of the point $\xi$, so we refer to it as $G_{\sigma}$. Let $\mathfrak{g}_{\sigma}$ be the Lie algebra of $G_{\sigma}$ and let $\mathfrak{t}_{\sigma} = \mathfrak{z}(\mathfrak{g}_{\sigma})\subseteq \mathfrak{t}$ be its centre. The span of $\sigma$ in $\mathfrak{g}^*$ is the subspace of points $(\mathfrak{g}^*)^{G_{\sigma}}$ fixed under the coadjoint action of $G_{\sigma}$, which is identified with $\mathfrak{t}_{\sigma}^*$ via the inner product on $\mathfrak{g}$.  Let $T_{\sigma}\subseteq T$ be the torus with Lie algebra $\mathfrak{t}_{\sigma}$.

We may use the action of $G$ to define a new action of $T_{\sigma}$ on $(\Phi)^{-1}(\Sigma_{\sigma})$ by letting
\begin{equation}\label{thimm action}
  t \ast m = (g^{-1}tg)\cdot m
\end{equation}
for all $m\in (\Phi)^{-1}(\Sigma_{\sigma})$ and $t\in T_{\sigma}$. Here $g\in G$ is an element such that $Ad_g^*\Phi(m) \in \sigma$, and $(g^{-1}tg)$ is acting on $M$ as an element of $G$. One checks that this action is independent of the choice of $g$ in the coset $G_{\sigma}g$ (since $T_{\sigma}$ is contained in the centre of $G_{\sigma}$). Note that the new action of $T_{\sigma}$ commutes with the action of $G$. 


In \cite[Theorem 3.4]{gs1}, Guillemin and Sternberg observed that if $\sigma$ is maximal (i.e. $\sigma = (\mathfrak{t}_+^*)^{\text{int}}$), then the new action of $T_{\sigma} = T$ on $\Phi^{-1}(\Sigma_{\sigma})$ is Hamiltonian, generated by $s\circ\Phi$. More generally, if $p_{\sigma} \colon \mathfrak{t}^* \rightarrow \mathfrak{t}_{\sigma}^*$ is the projection dual to the inclusion $\mathfrak{t}_{\sigma} \subseteq \mathfrak{t}$, then we have the following proposition (see \cite[Proposition 3.4]{woodward} for a proof).

\begin{proposition}[Guillemin and Sternberg's action coordinates]\label{pabiniak 2} The new action of $T_{\sigma}$ on $\Phi^{-1}(\Sigma_{\sigma})$ defined by equation \eqref{thimm action} is Hamiltonian and
$$p_{\sigma}\circ s\circ\Phi:\Phi^{-1}(\Sigma_{\sigma})\rightarrow \mathfrak{t}_{\sigma}^*$$
is a momentum map for this action.

\end{proposition}

In what follows, we identify $\mathfrak{t}_{\sigma}^*$ with the subspace of $\mathfrak{t}^*$ spanned by $\sigma$ so that $p_{\sigma} \circ s\circ\Phi = s\circ\Phi$. 

\section{Thimm's Trick}\label{s:thimms trick}

\subsection{Thimm's trick}

A completely integrable system on a symplectic manifold $(M,\omega)$ of dimension $2n$ is a set of $n$ Poisson commuting functions $f_1, \ldots, f_n \in C^{\infty}(M)$ such that the map 
	\begin{equation}
		(f_1, \ldots, f_n) : M \rightarrow \mathbb{R}^n
	\end{equation}
	is a submersion on an open dense subset of $M$.  A classical problem in the study of integrable systems is the construction of completely integrable systems on a Hamiltonian $G$-manifold from collective functions, where a function on a Hamiltonian $G$-manifold $(M,\omega,\Phi)$ is \emph{collective} if it is of the form $\Phi^*f$, for some $f\in C^{\infty}(\mathfrak{g}^*)$.  Since $\Phi$ is Poisson, if two functions $f,g\in C^{\infty}(\mathfrak{g}^*)$ commute with respect to the Kostant-Kirillov-Souriau Poisson bracket, then their pullbacks  $\Phi^*f$, $\Phi^*g$ commute with respect to the Poisson bracket on $M$.  
	
	Thus one may construct a completely integrable system of collective functions on $M^{2n}$ by constructing $n$ independent Poisson commuting functions $f_1, \ldots ,f_n \in C^{\infty}(\mathfrak{g}^*)$. By Chevalley's theorem, one can find $\text{rank}(G)$ independent functions in the ring of $Ad^*$-invariant functions, $C^{\infty}(\mathfrak{g}^*)^G$, and these functions Poisson commute, but it is often the case that $\text{rank}(G)< n$.  In order to find additional independent Poisson commuting functions on $\mathfrak{g}^*$, Thimm proved the following proposition \cite[Proposition 4.1]{thimm}.
	
	\begin{proposition}[Thimm's Trick]\label{thimm}  Let $G$ be a compact connected Lie group and let $\mathfrak{h}_1$, $\mathfrak{h}_2$ be two subalgebras of $\mathfrak{g}$ with dual projection maps $p_i: \mathfrak{g}^* \rightarrow \mathfrak{h}_i^*$, $i = 1,2$. If $h_1 \in C^{\infty}(\mathfrak{h}_1^*)^{H_1}$ and $h_2 \in C^{\infty}(\mathfrak{h}_2^*)$, then the Poisson bracket of $p_1^*h_1$ and $p_2^*h_2$ vanishes identically on $\mathfrak{g}^*$ if 
	\begin{equation}\label{condition}
		[\mathfrak{h}_1, \mathfrak{h}_2] \subseteq \mathfrak{h}_1.
	\end{equation}
	This holds in particular if $\mathfrak{h}_2 \subseteq \mathfrak{h}_1$ or $[\mathfrak{h}_1,\mathfrak{h}_2] = \{0\}$.
		
	\end{proposition}

	Thus one can find additional Poisson commuting functions on $\mathfrak{g}^*$ by pulling back invariant functions from subalgebras that satisfy (\ref{condition}). This construction of additional commuting functions is \emph{Thimm's trick}\footnote{There is a minor confusion in the literature between Thimm's trick and Guillemin and Sternberg's action coordinates.  Thimm's paper \cite{thimm} is solely concerned with constructing integrable systems whereas natural action coordinates for systems constructed by Thimm's trick originate in \cite{gs1}.}.
	
	\begin{remark} The statement of Proposition \ref{thimm} differs from the statement of \cite[Proposition 4.1]{thimm} in two ways. First, we note that it is only necessary for the function $h_1 \in C^{\infty}(\mathfrak{h}_1^*)$ to be $H_1$-invariant. This is implicit in the proof provided in \cite{thimm} which does not use the assumption that $h_2$ is $H_2$-invariant.  Second, it is not necessary to require $G$ to be compact (it is only assumed in \cite{thimm} that the subalgebras $\mathfrak{h}_1$ and $\mathfrak{h}_2$ are nondegenerate).
	\end{remark}
	
	\begin{remark}
		In the earlier paper \cite{trofimov} -- also concerned with the construction of collective completely integrable systems -- Trofimov also used chains of subalgebras $\mathfrak{h}_2 \subseteq \mathfrak{h}_1$ to find additional collective integrals. There are other constructions of additional commuting functions on $\mathfrak{g}^*$ that may give collective integrable systems such as the method of argument shifting \cite{manakov, mf}.
	\end{remark}
	
	If the coordinates of a completely integrable system $f_1, \ldots , f_n$ generate Hamiltonian $S^1$-actions, then the map $(f_1, \ldots , f_n)$ is said to provide \emph{action coordinates} for the integrable system.  In \cite[p. 119]{gs1}, Guillemin and Sternberg essentially applied Proposition \ref{pabiniak 2} to show that there are natural action coordinates for a collective completely integrable system on a Hamiltonian $G$-manifold $(M,\omega,\Phi)$ that has been constructed by Thimm's trick from a chain of subalgebras 
	$$ \mathfrak{h}_1 \subseteq \cdots \subseteq \mathfrak{h}_d$$
	(although they make the extra assumption that the principal stratum corresponding to $M$ is $(\mathfrak{t}_+^*)^{\text{int}}$).
	In general, systems of commuting collective functions constructed by Thimm's trick may not give a completely integrable system on $M$ (i.e. there are fewer than $n$ independent functions) and the principal stratum corresponding to $M$ may not be maximal. Nevertheless, it is still possible to construct a continuous map that generates a Hamiltonian torus action on an open dense subset as we will now see.
		
\begin{proposition}\label{thimm 1} Let $(M,\omega,\Phi)$ be a connected Hamiltonian $G$-manifold and suppose that  $\mathfrak{h}_1$, $\mathfrak{h}_2$ are subalgebras of $\mathfrak{g}$ with corresponding connected subgroups $H_1$ and $H_2$, such that 
$$ [  \mathfrak{h}_1, \mathfrak{h}_2] \subseteq \mathfrak{h}_1.$$
Let $\sigma_1$ be the principal stratum of a Weyl chamber $\mathfrak{t}_{1,+}^*\subseteq \mathfrak{h}_1^*$ corresponding to the induced action of $H_1$ on $M$. The action of $H_2$ on $M$ leaves $(p_1\circ \Phi)^{-1}(\Sigma_{\sigma_1})$ invariant and commutes with the new action of $T_{\sigma_1}$ defined there.
\end{proposition}

\begin{proof}
	Let $\sigma_1$ be the principal stratum of $\mathfrak{t}_{1,+}^*$ corresponding to $M$. Let $\Sigma_{\sigma_1} = H_1 \cdot \sigma_1$ and let $\mathcal{U} = (p_1\circ \Phi)^{-1}(\Sigma_{\sigma_1})$.  Identify $t_{\sigma_1}^*$ with the span of $\sigma_1$ in $\mathfrak{h}_1^*$ and fix $X_1 \in \mathfrak{t}_{\sigma_1}$.
	
	 For $\varepsilon >0 $, let 
	 $$(\sigma_1)_{\varepsilon} = \sigma_1 \setminus \bigcup_{\tau < \sigma_1} B_{\varepsilon}(\tau)$$
	 where $\tau < \sigma_1$ are the strata of $\mathfrak{t}_{1,+}^*$ contained in $\overline{\sigma_1}\setminus \sigma_1$ and  
	 $$ B_{\varepsilon}(\tau) = \left\{ \xi \in \mathfrak{t}_{\sigma_1}^* \colon d(\xi,\tau) < \varepsilon \right\}$$
	 is the $\varepsilon$-neighbourhood of $\tau$ in the subspace $\mathfrak{t}_{\sigma_1}^*$ (with respect to the metric induced by the inner product).
		 
	Since $s\colon \Sigma_{\sigma_1} \rightarrow \sigma$ is smooth, the function $\langle s,X_1\rangle\colon \Sigma_{\sigma_1} \rightarrow \mathbb{R}$ is smooth. Since $\Sigma_{\sigma_1}$ is submanifold of $\mathfrak{h}_1^*$, the restriction of $\langle s,X_1\rangle$ to the closed set $H_1 \cdot (\sigma_1)_{\varepsilon} \subseteq \Sigma_{\sigma_1}$ is smooth as a function on a closed subset of $\mathfrak{h}_1^*$.  By the Whitney Extension Theorem, there exists a smooth function $F_{\varepsilon} \in C^{\infty}(\mathfrak{h}_1^*)$ such that 
	$$F_{\varepsilon}\vert_{H_1 \cdot (\sigma_1)_{\varepsilon}} = \langle s,X_1\rangle.$$
	Since $H_1$ is compact and $\langle s,X_1\rangle$ is $H_1$-invariant, we can take $F_{\varepsilon}$ to be $H_1$-invariant.
	
	 Since  $F_{\varepsilon} \in C^{\infty}(\mathfrak{h}_1^*)^{H_1}$ and $[\mathfrak{h}_1, \mathfrak{h}_2] \subseteq \mathfrak{h}_1$, we may apply Proposition \ref{thimm} to conclude that for all $X_2 \in \mathfrak{h}_2$,
	 $$ \left\{ \langle p_2\circ \Phi,X_2\rangle, F_{\varepsilon} \circ p_1 \circ \Phi \right\}_M = \left\{ \langle p_2,X_2\rangle, F_{\varepsilon} \circ p_1 \right\}_{\mathfrak{g}^*} = 0.$$
	 Let $\mathcal{U}_{\varepsilon} = (p_1\circ \Phi)^{-1}\left(H_1\cdot (\sigma_1)_{\varepsilon}^{\text{rel-int}}\right)$. The sets $\mathcal{U}_{\varepsilon} $ are open in $M$ and 
	 $$ \mathcal{U} = \bigcup_{\varepsilon>0} \mathcal{U}_{\varepsilon}.$$	
	Thus we have shown that for all $\varepsilon>0$, 
	$$\left\{ \langle p_2 \circ \Phi,X_2\rangle,\langle s \circ p_1 \circ \Phi, X_1\rangle\right\}\vert_{\mathcal{U}_{\varepsilon}} = \left\{ \langle p_2 \circ \Phi,X_2\rangle,F_{\varepsilon}\circ p_1 \circ \Phi\right\}\vert_{\mathcal{U}_{\varepsilon}} = 0.$$
	so
	$$\left\{ \langle p_2 \circ \Phi,X_2\rangle,\langle s \circ p_1 \circ \Phi, X_1\rangle\right\}_{\mathcal{U}} = 0.$$
	Since $H_2$ is connected and the action of $H_2$ is Hamiltonian, generated by $p_2\circ \Phi$, it follows  that the action of $H_2$ on $M$ leaves $(p_1\circ \Phi)^{-1}(\Sigma_{\sigma_1})$ invariant and commutes with the new action of $T_{\sigma_1}$ defined there.
\end{proof} 

Let $\sigma_2$ be the principal stratum of $\mathfrak{t}_{2,+}^*$ corresponding to the induced $H_2$-action on $M$, which is generated by $p_2 \circ\Phi$.  A new action of the torus $T_{\sigma_2}$ on $(p_2\circ \Phi)^{-1}(\Sigma_{\sigma_2})$ is defined as in equation \eqref{thimm action} via the action of elements of $H_2$. It follows by the preceding proposition that this action leaves $(p_1\circ \Phi)^{-1}(\Sigma_{\sigma_1})$ invariant and commutes with the new action of $T_{\sigma_1}$. Thus we have proven the following corollary.

\begin{corollary}\label{thimm torus}
	If $[\mathfrak{h}_1,\mathfrak{h}_2]\subseteq \mathfrak{h}_1 \text{ or } \mathfrak{h}_2$, then the new $T_{\sigma_1}$-action generated by $s\circ p_1\circ\Phi$ and the new $T_{\sigma_2}$-action generated by $s\circ p_2\circ\Phi$ commute on the open, dense subset where they are both defined, $$(p_1\circ\Phi)^{-1}(\Sigma_{\sigma_1})\cap(p_2\circ\Phi)^{-1}(\Sigma_{\sigma_2}).$$
\end{corollary}


Let $\mathfrak{h}_1, \ldots, \mathfrak{h}_d$ be subalgebras of $\mathfrak{g}$ such that for all $1 \leq i,j \leq d$,
\begin{equation}\label{condition of thimm}
	[\mathfrak{h}_i,\mathfrak{h}_j] \subseteq \mathfrak{h}_i \text{ or } \mathfrak{h}_j
\end{equation}
and let $(M,\omega,\Phi)$ be a connected Hamiltonian $G$-manifold.  For each $1 \leq k \leq d$, fix maximal tori $T_k$ and Weyl chambers $\mathfrak{t}_{k,+}^* \subseteq \mathfrak{t}_k^*$.  The maps $p_k\circ \Phi$ are momentum maps for the induced actions of $H_k$ on $M$ and by Proposition \ref{principal stratum} there are principal strata $\sigma_k \subseteq \mathfrak{t}_{k,+}^*$ corresponding to each of the Hamiltonian $H_k$-actions on $M$.  For each $k$, let $U_k$ denote the connected open dense subset $\left( p_k \circ \Phi\right)^{-1}(\Sigma_{\sigma_k})$ where $\Sigma_{\sigma_k} = H_k \cdot \sigma_k$.  For each $k$, let $T_{\sigma_k}$ be the torus with Lie algebra $\mathfrak{t}_{\sigma_k}$.

Let $F$ be the composition of $\Phi$ with the map
\begin{equation}\label{continuous map}
	\left(s\circ p_1  ,  \ldots,  s\circ p_d \right)\colon \mathfrak{g}^*\rightarrow  \mathfrak{t}_1^* \oplus \cdots \oplus \mathfrak{t}_d^*.
\end{equation} 
We say that $F$ is \emph{constructed by Thimm's trick with Guillemin and Sternberg's action coordinates} from the subalgebras $\mathfrak{h}_1, \ldots , \mathfrak{h}_d$.  The map $F$ is continuous since the sweeping maps are continuous, but in general the map $F$ is not smooth.

By Corollary \ref{thimm torus}, the set
\begin{equation}\label{thimm: definition of U}
  \mathcal{U} = F^{-1}(\sigma_1 \times \cdots \times \sigma_d)   = \bigcap_{1\leq k \leq d} U_k
\end{equation} 
is invariant under the new actions of the tori $T_{\sigma_k}$, and the new actions of the tori $T_{\sigma_k}$ commute on $\mathcal{U}$.  Thus the new actions define an action of a big torus $T' = T_{\sigma_1} \times \cdots \times T_{\sigma_d}$
on $\mathcal{U}$ and this action is Hamiltonian with momentum map $F\vert_{\mathcal{U}} \colon \mathcal{U} \rightarrow \mathfrak{t}_1^* \oplus \cdots \oplus \mathfrak{t}_d^*$.  

\begin{lemma}\label{connected} The open dense set $\mathcal{U}$ is connected.
	
\end{lemma}

\begin{proof} By equation \eqref{thimm: definition of U}, $\mathcal{U}$ is the finite intersection of the sets $U_k\subseteq M$.  By Proposition \ref{preimage of open face}, the complement of $\mathcal{U}$ is a finite union of closed sets that are each contained in a locally finite union of submanifolds of codimension at least 2 in $M$, so $\mathcal{U}$ is connected. 
\end{proof}

Thus, given a connected Hamiltonian $G$-manifold $(M,\omega,\Phi)$ and subalgebras $\mathfrak{h}_1, \ldots, \mathfrak{h}_d$ that pairwise satisfy \eqref{condition of thimm}, 
 $$\left(\mathcal{U},\omega\vert_{\mathcal{U}},F\vert_{\mathcal{U}}\right)$$
is a connected Hamiltonian $T'$-manifold, where $T' = T_{\sigma_1}\times \cdots \times T_{\sigma_d}$.  

\begin{example}[The classical Gelfand-Zeitlin systems]\label{gz system}
 Let $(\mathcal{O},\omega)$ be a $U(n)$ coadjoint orbit  equipped with the Kostant-Kirillov-Souriau symplectic form. The induced action of a subgroup $U(n-1) \subseteq U(n)$ is Hamiltonian, and the restriction of the projection $p\colon \mathfrak{u}(n)^* \rightarrow \mathfrak{u}(n-1)^*$ to $\mathcal{O}$ is a momentum map for the induced action.  The classical Gelfand-Zeitlin system on $\mathcal{O}$ is the torus action constructed by Thimm's trick with Guillemin and Sternberg's action coordinates from a chain of subalgebras 
 $$ \mathfrak{u}(1) \subseteq \cdots \subseteq \mathfrak{u}(n-1).$$
Since $\mathcal{O}$ is compact, the continuous map 
 $$F \colon\mathcal{O} \rightarrow  \mathfrak{t}_1^* \oplus \cdots  \oplus \mathfrak{t}_{n-1}^*.$$
is proper.  By definition, $\mathcal{U} = F^{-1}\left( \sigma_1 \times \cdots \times \sigma_{n-1}\right)$, so the restriction $ F\vert_{\mathcal{U}}$ is proper as a map to the convex set 
 $$\sigma_1 \times \cdots \times \sigma_{n-1} \subseteq \mathfrak{t}_1^* \oplus \cdots \oplus \mathfrak{t}_{n-1}^*.$$
  Combining this with Lemma \ref{connected}, we have shown that $\left(\mathcal{U},\omega\vert_{\mathcal{U}},F\vert_{\mathcal{U}}\right)$ is a proper Hamiltonian $T'$-manifold.  Combining this with the fact that the $T'$ action on $\mathcal{U}$ is effective and completely integrable \cite{gs1,pabiniak}, we have shown that $\left(\mathcal{U},\omega\vert_{\mathcal{U}},F\vert_{\mathcal{U}}\right)$ is a proper toric $T'$-manifold.
 
 The construction of the classical Gelfand Zeitlin systems on $SO(n)$ coadjoint orbits is completely analogous, except that one considers a chain of subalgebras 
  $$ \mathfrak{so}(2) \subseteq \cdots \subseteq \mathfrak{so}(n-1)$$
 instead. 
\end{example}

More generally, if $(\mathcal{U},\omega\vert_{\mathcal{U}},F\vert_{\mathcal{U}})$ is a proper Hamiltonian $T'$-manifold, then it follows by the convexity theorem for proper Hamiltonian torus actions (cf. \cite[Theorem 30]{kb}), that 
\begin{enumerate}[(a)]
	\item $F(\mathcal{U})$ is convex,
	\item the fibres of $F\vert_{\mathcal{U}}$ are connected, and 
	\item the domain and codomain restricted map $F\vert_{\mathcal{U}} \colon \mathcal{U} \rightarrow F(\mathcal{U})$ is open (with respect to the subspace topologies on $\mathcal{U}$ and $F(\mathcal{U})$).
\end{enumerate}

For instance, if $F\colon M \rightarrow \mathfrak{t}_1^* \oplus \cdots \oplus \mathfrak{t}_d^*$ is proper (e.g. if $M$ is compact), then $F\vert_{\mathcal{U}}$ is proper as a map to the convex set $\sigma_1 \times \cdots \times \sigma_n$.  It is natural to wonder whether properties (a), (b) and (c) are also true for the map $F\colon M \rightarrow \mathfrak{t}_1^* \oplus \cdots \oplus \mathfrak{t}_d^*$ if $F$ is proper. This prompts the following lemma.

\begin{lemma}\label{marginal connectedness} Let $X$ be a Hausdorff topological space and let $f\colon X \rightarrow \mathbb{R}^n$ be a continuous proper map. Suppose $S$ is a dense subset of $X$ saturated\footnote{A subset $S\subseteq X$ is saturated by a map $f\colon X\rightarrow Y$ if it is a union of level sets of $f$.} by $f$ such that  $f(S)$ is convex, the fibres of $f|_S$ are connected, and $f|_S\colon S \rightarrow f(S)$ is open. Then the fibres of $f$ are  connected.
\end{lemma}

\begin{proof}
Since the fibres of $f|_S$ are connected and $S$ is saturated by $f$, it remains to show that the fibres $f^{-1}(x)$, $x \in f(X\setminus S)$ are connected.

Fix $x \in f(X\setminus S)$ and let $F = f^{-1}(x)$. Suppose that $F = U \cup V$ for open sets $U,V \subseteq F$. Since $U$ and $V$ are open with respect to the subspace topology on $F$, there are open sets $U',V' \subseteq X$ with $U = U' \cap F$, $V = V'\cap F$.  Our aim is to show that $U\cap V \neq \emptyset$.

First, we claim that there is a $\delta > 0$ such that $f^{-1}(B_{\delta}(x)) \subseteq U' \cup V'$, where $B_{\delta}(x)$ is the open ball of radius $\delta$ centred at $x$. If not then, for all $n \geq 1$, there is a $y_n \in f^{-1}(B_{1/n}(x))$ such that $y_n \not \in U' \cup V'$. Since $f$ is proper, $f^{-1}(\overline{B}_{1}(x))$ is compact in $X$. Thus there is a convergent subsequence $y_{n_k} \rightarrow y$, and $y\in F$ by continuity. Since $y_n$ is contained in the closed set $X\setminus(U'\cup V')$, so is the limit $y$, which contradicts our assumption that $F = U \cup V$.

For every $\gamma>0$, the set $B_{\gamma}(x) \cap f(S)$ is convex and non-empty. Further,
$$ A_{\gamma} =B_{\gamma}(x)\cap f(U'\cap S) \mbox{ and } C_{\gamma} =B_{\gamma}(x)\cap f(V'\cap S)$$
are non-empty, open subsets of $f(S)$, since $f|_S$ is open as a map to $f(S)$ and $S$ is dense. By the previous paragraph, $B_{\gamma}(x) \cap f(S) \subseteq A_{\gamma} \cup C_{\gamma}$ for all positive $\gamma<\delta$. Thus since $B_{\gamma}(x) \cap f(S)$ is connected and non-empty, $A_{\gamma} \cap C_{\gamma} \neq \emptyset$ for all positive $\gamma< \delta$.

Fix $1/N < \delta$. For every $n \geq N$, we have shown there is an element $x_n \in A_{1/n}\cap C_{1/n}$. Thus the intersections $U' \cap f^{-1}(x_n)$ and $V' \cap f^{-1}(x_n)$ are both non-empty. Since $f^{-1}(x_n) \subseteq U' \cup V'$. and $f^{-1}(x_n)$ is connected, there exists a $y_n \in U' \cap V' \cap f^{-1}(x_n)$.  Once again, by properness of $f$, we can find a convergent subsequence $y_{n_k}$ with limit $y\in F$.  

It follows that $U \cap V \neq \emptyset.$ Suppose $y\not\in V$ (and thus, $y\in U$).  Since $y_n \in V'$, $y \in \overline{V'}$. Observe that $\overline{V'} \cap F = \overline{V' \cap F}  =\overline{V} $ (where the second and third closures are taken in the subspace topology) so $y \in \overline{V}$.  This contradicts $U \cap V = \emptyset$ since $U$ is an open neighbourhood of $y$ in $F$. 
\end{proof}
%

%
\begin{remark}
	Note that Lemma \ref{marginal connectedness} is false if various assumptions are dropped.  For example,
	\begin{itemize}
		\item Let $X$ be the sphere of radius 1 in $\mathbb{R}^3$ and let $f\colon X \rightarrow \mathbb{C}$ be the map $f(x,y,z) = e^{\pi i (z+1)}$. Then
		\begin{itemize}
			\item $f$ is a proper continuous map.
			\item The set $S = X \setminus \left\{ (0,0,1),(0,0,-1)\right\}$ is dense in $X$ and saturated by $f$.
			\item The restriction $f\vert_S\colon S \rightarrow f(S)$ is open.
			\item The fibres of $f\vert_S$ are connected
		\end{itemize}
		But $f(S)$ not convex and the fibre $f^{-1}(1)$ is not connected.
		\item Let $X$ be the unit square $[0,1]\times [0,1]$ minus the set $\{1\} \times (0,1)$. Let $f\colon X \rightarrow \mathbb{R}$ be projection to the first coordinate. Then 
		\begin{itemize}
			\item $f$ is a continuous map.
			\item The set $S = [0,1) \times [0,1]$ is dense in $X$ and saturated by $f$.
			\item The restriction $f\vert_S\colon S \rightarrow f(S)$ is open.
			\item The fibres of $f\vert_S$ are connected.
		\end{itemize}
		But $f$ is not proper and the fibre $f^{-1}(1)$ is not connected.
	\end{itemize}
\end{remark}

We can now prove Theorem \ref{main theorem}.

\begin{proof}[Proof of Theorem 1:] (1) By Lemma \ref{connected}, $\mathcal{U}$ is connected. Since $F$ is proper, $F\vert_{\mathcal{U}}$ is proper as a map to the convex set $\sigma_1 \times \cdots \times \sigma_d$. Therefore $(\mathcal{U}, \omega\vert_{\mathcal{U}},F\vert_{\mathcal{U}})$ is a proper Hamiltonian $T'$-manifold.

\vspace{.5em}

\noindent (2) By (1) and the convexity theorem for proper Hamiltonian torus actions (cf. \cite[Theorem 30]{kb}), $F(\mathcal{U})$ is convex. Since $\mathcal{U}$ is dense in $M$ and $F$ is a proper continuous map, $F(M) = F(\overline{\mathcal{U}}) = \overline{F(\mathcal{U})}$. Thus $F(M)$ is convex.

\vspace{.5em}

\noindent (3) We check the hypotheses of Lemma \ref{marginal connectedness}. $F$ is continuous and by assumption, $F$ is proper. The set $\mathcal{U}$ is dense in $M$ and by definition it is saturated by $F$. By (1) and the convexity theorem for proper Hamiltonian torus actions (cf. \cite[Theorem 30]{kb}), the restriction $F\vert_{\mathcal{U}}$ satisfies conditions (a), (b), and (c). These are precisely the remaining hypotheses of Lemma \ref{marginal connectedness}, so the result follows. 
\end{proof}
 
\begin{remark} The non-abelian convexity theorem for Hamiltonian group actions says that if $(M,\omega,\Phi)$ is a Hamiltonian $G$-manifold with $\Phi$ proper, then the \emph{momentum set}
\begin{equation*}
	\square(M) := s\circ \Phi(M) = \Phi(M) \cap \mathfrak{t}_+^*	
\end{equation*}
is a convex, locally polyhedral set, the fibres of $s \circ \Phi$ are connected, and the map $s\circ \Phi$ is open as a map to its image\footnote{The first fact is standard in the literature (see e.g. \cite[Theorem 1.1]{lmtw}). The second fact is equivalent to fibre connectedness for the map $\Phi$ (see the argument in \cite[p. 256]{lmtw}). Openness of $s\circ\Phi$ as a map to its image is described in \cite[Theorem 1.2]{lmtw}).}.  Thus if $F$ is a map constructed by Thimm's trick with Guillemin and Sternberg's action coordinates from one subalgebra $\mathfrak{h}$, Theorem \ref{main theorem} reduces to the non-abelian convexity theorem.  Conversely, Theorem \ref{main theorem} does not follow directly from the non-abelian convexity theorem for Hamiltonian group actions; a product of maps with convex images and connected fibres does not necessarily have a convex image and connected fibres.
\end{remark}

\begin{remark} We were unable to determine in the context of Theorem \ref{main theorem} whether $F$ is open as a map to its image (property (c) above).  It is known that $F$ is open as a map to its image in the special case when $F = s\circ \Phi$ by the non-abelian convexity theorem for Hamiltonian group actions \cite[Theorem 1.2]{lmtw}.	
\end{remark}

\subsection{Symplectic Reduction and Thimm's Trick}\label{s:symplectic reduction}

Suppose that $(M, \omega,\Phi)$ is a connected Hamiltonian $G$-manifold and $F$ is a map constructed by Thimm's trick with Guillemin and Sternberg's action coordinates from subalgebras $\mathfrak{h}_1, \ldots ,\mathfrak{h}_d$ of $\mathfrak{g}$ as in the preceding section. 

Let $K$ be a closed, connected subgroup of $G$. The induced action of $K$ on $M$ is Hamiltonian with momentum map $\Phi_K = p_K \circ \Phi$.  The \emph{reduction} of $M$ at a level $\mu \in \mathfrak{k}^*$ is the quotient space $M_{\mu} = \Phi_K^{-1}(K\cdot\mu)/K$, where $K\cdot\mu$ is the coadjoint orbit of $K$ through $\mu$.  If $\mu$ is a regular value of $\Phi_K$, then this quotient space is a smooth manifold or an orbifold with a Marsden-Weinstein symplectic structure. Proposition \ref{reduction 2} is purely topological and will hold in the general setting where $\mu$ may be a critical value. Let $\pi\colon \Phi_K^{-1}(K\cdot\mu)\rightarrow M_{\mu}$ be the quotient projection map.

\begin{lemma}\label{reduction 1}
  If $[\mathfrak{h}_k,\mathfrak{k}]\subseteq \mathfrak{h}_k$ for all $1 \leq k\leq d$, then $F$ induces a map $\widetilde F$ on $M_{\mu}$ and the restriction of $\widetilde F$ to $\pi(\mathcal{U}\cap \Phi_K^{-1}(K\cdot\mu))$ generates a Hamiltonian torus action.
\end{lemma}

Note that the set $\mathcal{U} \cap \Phi_K^{-1}(K\cdot\mu)$ could be empty.

\begin{proof}
  By Proposition \ref{thimm 1}, the actions of $K$ preserves $\mathcal{U}$ and commutes with the action of the big torus $T'$ defined there. The map $F$ is $K$-invariant on $M$ by continuity (since $\mathcal{U}$ is dense in $M$). Thus $F$ induces a map $\widetilde F$ on $M_{\mu}$ and  $\widetilde F$ is a momentum map for the induced $T'$-action on $\pi(\mathcal{U}\cap \Phi_K^{-1}(K\cdot \mu))$ if it is non-empty (see \cite[Example 1.11 and Lemma 3.2]{sl}). 
\end{proof}

\begin{proposition}\label{reduction 2}
  Suppose that either $F$ or $\Phi_K$ is proper and $[\mathfrak{h}_k,\mathfrak{k}] \subseteq \mathfrak{h}_k$ for all $1\leq k\leq d$. Then the induced map $\widetilde F\colon M_{\mu} \rightarrow \mathfrak{t}_1^* \oplus \cdots \oplus \mathfrak{t}_d^*$ has connected fibres and $\widetilde F(M_{\mu})$ is convex.
\end{proposition}

\begin{proof}
  By Corollary \ref{thimm torus}, the map
  \begin{equation}
    ( F,s\circ\Phi_K)\colon M \rightarrow \left(\mathfrak{t}_1^* \oplus \cdots \oplus \mathfrak{t}_d^*\right) \oplus \mathfrak{t}_{K}^*
  \end{equation}
  is a momentum map for the new action of $T' \times T_{\sigma_K}$ on 
  \begin{equation}
    \mathcal{U} \cap \left(s\circ\Phi_K\right)^{-1}(\sigma_K) =  ( F,s\circ\Phi_K)^{-1}\left(\sigma_1\times \cdots \times \sigma_d \times \sigma_K\right).
  \end{equation}
  Since one of $F$ or $ \Phi_K$ is proper, the continuous map $( F,s\circ\Phi_K)$ is proper. Thus by Theorem \ref{main theorem}, the image of $( F,s\circ\Phi_K)$ is convex and the fibres are connected.  The image $\widetilde F(M_{\mu})$ is the projection to $\mathfrak{t}_1^* \oplus \cdots \oplus \mathfrak{t}_d^* \subseteq \left(\mathfrak{t}_1^* \oplus \cdots \oplus \mathfrak{t}_d^*\right) \oplus \mathfrak{t}_{K}^*$ of the intersection
  \begin{equation}
    ( F,s\circ\Phi_K)(M) \cap \left(\sigma_1\times \cdots \times \sigma_d \times \left\{ \mu \right\}\right)
  \end{equation}
  which is convex, so $\widetilde F(M_{\mu})$ is convex.  The map $\widetilde F$ has connected fibres since $( F,s\circ\Phi_K)$ has connected fibres that are preserved by the action of $K$.  
\end{proof}

\section{Branching and Gelfand-Zeitlin systems}

\subsection{The branching cone}

We begin by recalling the branching cone, as described in \cite{baldoni-vergne}.

Let $G$ be a compact, connected Lie group. With respect to the trivialization of the cotangent bundle $T^*G$ by left-invariant vector fields the cotangent lifts of the left and right actions of $G$ on itself are given by 
	\begin{equation}
		\mathcal{L}_{g'}(g,\xi) = (g'g,\xi),\,\text{ and } \mathcal{R}_{g'}(g,\xi) = (g(g')^{-1},Ad_{g'}^*\xi).
	\end{equation}
	and these actions are Hamiltonian with respect to the canonical symplectic form on $T^*G$, generated by momentum maps
	\begin{equation}
		\Phi_{\mathcal{L}}(g,\xi) = -Ad_g^*\xi \text{ and } \Phi_{\mathcal{R}}(g,\xi) = \xi.
	\end{equation}
	Let $K\leq G$ be a closed, connected subgroup, and consider $T^*G$ as a Hamiltonian $K \times G$-manifold, where 
	\begin{equation}
		(k,g')\cdot (g,\xi) = \mathcal{R}_k\mathcal{L}_{g'}(g,\xi).
	\end{equation}
	This action is generated by the momentum map
	\begin{equation}
		(g,\xi) \mapsto \left( p_G^K\circ \Phi_{\mathcal{R}}(g,\xi),\Phi_{\mathcal{L}}(g,\xi)\right) = \left( p_G^K(\xi),-Ad_{g}^*\xi\right)
	\end{equation}
	where $p_G^K\colon \mathfrak{g}^* \rightarrow \mathfrak{k}^*$ is the projection dual to the inclusion $i\colon\mathfrak{k} \rightarrow \mathfrak{g}$. Following \cite{baldoni-vergne}, we consider the modified map 
	\begin{equation}
		\Psi\colon T^*G \rightarrow \mathfrak{k}^*\times \mathfrak{g}^* ,\, \Psi(g,\xi) = \left(p_G^K(\xi), Ad_{g}^*\xi\right).
	\end{equation}
	
	If $\mathfrak{t}_{G,+}^*$ and $\mathfrak{t}_{K,+}^*$ are positive Weyl chambers then 
	$$\mathfrak{t}_{K,+}^*\times \mathfrak{t}_{G,+}^* \subseteq  \mathfrak{k}^*\times \mathfrak{g}^*$$
	is a positive Weyl chamber for $K\times G$.  The map $\Psi$ is proper so by the non-abelian convexity theorem for Hamiltonian group actions \cite[Theorem 1.1]{lmtw}, the set 
	\begin{equation}\label{branching cone}
		\begin{split}
			C_{K,G} & = \Psi(T^*G) \cap \left(\mathfrak{t}_{K,+}^*\times \mathfrak{t}_{G,+}^* \right)\\
				& = \left\{ \left(\pi_G^K(\xi),Ad^*_g\xi\right) \in  \mathfrak{t}_{K,+}^*\times \mathfrak{t}_{G,+}^* \colon (g,\xi ) \in G\times \mathfrak{g}^* \right\}\\
				& = \left\{ \left(\eta,\xi\right) \in \mathfrak{t}_{K,+}^*\times \mathfrak{t}_{G,+}^* \colon \eta \in p_G^K(\mathcal{O}_{\xi}) \right\},
		\end{split}
	\end{equation}
	is a convex polyhedral cone.  This set is called the \emph{branching cone} for the subgroup $K \leq G$ \cite{baldoni-vergne}. Since it is a polyhedral cone, $C_{K,G}$ is defined as a subset of $\mathfrak{t}_{K}^* \times \mathfrak{t}_{G}^*$ by a finite list of inequalities, 
	\begin{equation}\label{branching inequalities}
		\langle a_i,\xi\rangle + \langle b_i, \eta\rangle \leq \kappa_i, \, i = 1, \ldots, N
	\end{equation}
	for $a_i\in \mathfrak{t}_G$, $b_i\in \mathfrak{t}_K$, and $\kappa_i \in \mathbb{R}$, which are called \emph{branching inequalities}.

	Note that if we fix $\xi\in \mathfrak{t}_{G,+}^*$, then the momentum set of the $G$ coadjoint orbit through $\xi$, as a Hamiltonian $K$-manifold, is the set 
	\begin{equation}
		p_G^K(\mathcal{O}_{\xi})\cap \mathfrak{t}_{K,+}^* = \left\{ \eta \in \mathfrak{t}_{K,+}^* \colon \, \eta \in p_G^K(\mathcal{O}_{\xi})\right\}
	\end{equation}
	which can be identified with the projection to $\mathfrak{t}_K^*$ of the intersection $\left(\mathfrak{t}_K^*\times \{\xi\} \right) \cap C_{K,G}.$ Accordingly, the momentum set of $\mathcal{O}_{\xi}$ is equal to the subset of $\mathfrak{t}_K^*$ defined by the inequalities (\ref{branching inequalities}) with $\xi$ fixed.
	
	\begin{remark}
		If $U(n)$ is embedded in $U(n+1)$ as a subgroup of block diagonal matrices $\text{diag}(A,1)$ then the inequalities defining the branching cone $C_{U(n),U(n+1)}$ can be described as the classical interlacing inequalities for eigenvalues of principal submatrices of Hermitian matrices, as observed in \cite{gs1}. Similar inequalities describe the branching cone for pairs $SO(n) \leq SO(n+1)$ \cite{pabiniak1}. Inequalities defining the branching cone of a general pair $K \leq G$ were described in \cite{berenstein-sjamaar}. 
	\end{remark}

\subsection{Branching and Thimm's trick}
	
	Suppose we have a connected Hamiltonian $G$-manifold $(M,\omega,\Phi)$ along with a chain of subalgebras $
		\mathfrak{h}_1 \subseteq \cdots \subseteq \mathfrak{h}_d= \mathfrak{g}$ 
	with corresponding connected subgroups $H_k$, maximal tori $T_k$, and choices of positive Weyl chambers $\mathfrak{t}_{k,+}^* \subseteq \mathfrak{t}_k^*$. For each $k$, let $\sigma_k \subseteq \mathfrak{t}_{k,+}^*$ be the principal stratum corresponding to the induced Hamiltonian action of $H_k$ on $M$ and let $T_{\sigma_k}$ be the corresponding subtorus of $T_k$. We consider the map $F$ constructed by Thimm's trick from this chain,
	\begin{equation}\label{gs map}
		F = \left( s\circ p_d^1\circ \Phi , \ldots, s\circ p_d^{d-1}\circ\Phi,s\circ\Phi\right) \colon M \longrightarrow \mathfrak{t}_1^* \oplus  \cdots \oplus \mathfrak{t}_{d}^*,
	\end{equation}
	which generates a Hamiltonian action of the torus $T' = T_{\sigma_1}\times \cdots \times T_{\sigma_d}$ on the open dense subset $\mathcal{U}$.  The projections $p_j^i\colon \mathfrak{h}_j^* \rightarrow \mathfrak{h}_i^*$ satisfy the identities  $p_j^i = p_{k}^i \circ  p_j^{k}$ for all $i<k<j$.  For each $1\leq  k < d$, let 
	\begin{equation}\label{inequalities 1}
		\langle a_{i,k+1},\xi\rangle + \langle b_{i,k} ,\,\eta\rangle \leq \kappa_{i,k} ,\, i = 1, \ldots ,N_k
	\end{equation}
	be the inequalities defining the branching cone $C_{H_{k},H_{k+1}}$ as a subset of $\mathfrak{t}_{k}^* \times \mathfrak{t}_{k+1}^*$, where $a_{i,k+1} \in \mathfrak{t}_{k+1}$, $b_{i,k} \in \mathfrak{t}_{k}$ and $\kappa_{i,k} \in \mathbb{R}$. If $\Phi$ is proper, then by the non-abelian convexity theorem for Hamiltonian group actions \cite[Theorem 1.1]{lmtw} the momentum set $\square \colon= \Phi(M) \cap \mathfrak{t}_{d,+}^*$ is a convex, locally polyhedral set defined by a list of inequalities
	\begin{equation}\label{inequalities 2}
		\langle \alpha_j, \xi \rangle \leq \upsilon_j,\, j \in S
	\end{equation}
	where $\alpha_j \in \mathfrak{t}_{d}$ and $\upsilon_j \in \mathbb{R}$, and $S$ is a set indexing the inequalities (note that without the assumption that $M$ is compact, $S$ may be infinite).
	
	\begin{proposition}\label{branching system}
		Let $(M,\omega,\Phi)$ be a connected Hamiltonian $G$-manifold with $\Phi$ proper and let $\mathfrak{h}_1 \subseteq \cdots \subseteq \mathfrak{h}_{d} = \mathfrak{g}$ be a chain of subalgebras.  The image of the map $F$ is the convex, locally polyhedral subset of $\mathfrak{t}_1^*\oplus \cdots \oplus \mathfrak{t}_d^*$ defined by the inequalities (\ref{inequalities 1}) and (\ref{inequalities 2}). Furthermore, the fibres of $F$ are connected.
	\end{proposition}
	
	\begin{proof} Since $\Phi$ is proper and $\mathfrak{h}_d = \mathfrak{g}$, the map $F$ is proper, so it follows by Theorem \ref{main theorem} that $F(M)$ is convex and the fibres of $F$ are connected.  It remains to show that $F(M)$ is equal to the set $\Delta\subseteq \mathfrak{t}_1^* \oplus \cdots \oplus \mathfrak{t}_d^*$ defined by the inequalities (\ref{inequalities 1}) and (\ref{inequalities 2}).
		
		Suppose that $(\xi_1, \ldots , \xi_d) = F(m)$ for some $m\in M$.  Since $\xi_d = s\circ \Phi(m)$, the inequalities (\ref{inequalities 2}) are satisfied.  For each $1\leq k< d$, since  $p_d^{k} = p_{k+1}^{k} \circ p_{d}^{k+1} $, it is true that $\xi_{k} \in p_{k+1}^{k}(H_{k+1} \cdot p_d^{k+1}(\xi_d))$, so the inequalities (\ref{inequalities 1}) are satisfied. Thus $F(M) \subseteq \Delta$
		
		The proof that $\Delta \subseteq F(M)$ follows from a repeated application of the branching inequalities.
		
		Let $(\xi_1, \ldots , \xi_d) \in \Delta$. Since $\xi_{1}, \xi_{2}$ satisfy the inequalities (\ref{inequalities 2}), we know that $\xi_{1} \in p_{2}^{1}(H_2\cdot \xi_2)$. Thus there exists $h_{2} \in H_2$ such that $\xi_{1} = p_2^1(Ad_{h_2}^* \xi_{2})$. Similarly, for each $2\leq k < d$, since $\xi_{k},\xi_{k+1}$ satisfy the inequalities (\ref{inequalities 1}), for each $h_k\in H_k$, there exists  $h_{k+1} \in H_{k+1}$  such that $Ad_{h_{k}}^*\xi_{k} = p_{k+1}^{k}(Ad_{h_{k+1}}^*\xi_{k+1})$. 
		
		Continuing in this fashion, we find $h_d \in H_d$ such that $s\circ p_d^k(Ad^*_{h_d}\xi_d) = \xi_k$ for all $1\leq k <d$. Finally, since $\xi_d$ satisfies the inequalities (\ref{inequalities 2}) there exists $m\in M$ such that $ \Phi(m) = Ad^*_{h_d}\xi_d$ and it follows that $(\xi_1, \ldots , \xi_d) = F(m)$.	
	\end{proof}

\subsection{Smooth fibres and Gelfand-Zeitlin systems}\label{s: smooth fibres}

If $\mu$ is the momentum map for a completely integrable Hamiltonian torus action, then the connected components of the fibres of $\mu$ are isotropic tori that are generically Lagrangian. In comparison, if a map $F$ constructed by Thimm's trick with Guillemin and Sternberg's action coordinates generates a completely integrable torus action on the open dense subset $\mathcal{U}$ then very little is known about the fibres of $F$  in the complement of $\mathcal{U}$.  In this subsection we give a short proof that if the torus action generated by $F$ is completely integrable, then the fibres in the complement of $\mathcal{U}$ are smooth embedded submanifolds.

Recall that	a Hamiltonian $G$-manifold $M$ is \emph{multiplicity free} if the Poisson subalgebra of $G$-invariant functions, $C^{\infty}(M)^G$, is abelian.

\begin{proposition}\label{woodward characterization 1} \cite[Proposition A.1]{woodward1} Let $G$ be a compact, connected Lie group and let $(M,\omega,\Phi)$ be a connected Hamiltonian $G$-manifold with $\Phi$ proper\footnote{In \cite{woodward1} it is  assumed that $M$ is compact. This assumption can be removed using the non-abelian convexity theorem for proper momentum maps \cite[Theorem 1.1]{lmtw}.}. Then
	\begin{enumerate}
		\item $M$ is multiplicity free if and only $G$ acts transitively on $\Phi^{-1}(\mathcal{O}_{\xi})$ for every $\xi \in \Phi(M)$.
		\item If in addition $G$ acts locally freely on a dense set then $M$ is multiplicity free if and only if $\dim(M) = \dim(G) + \text{rank}(G)$.
	\end{enumerate}
\end{proposition}

In particular, if $G$ is a torus acting effectively on $M$ with principal orbit-type $(1)$, then $M$ is multiplicity free if and only if $\dim(M) = 2\dim(G)$. 

\begin{proposition}\label{gs completely integrable}
	Let $(M,\omega,\Phi)$ be a connected Hamiltonian $G$-manifold with $\Phi$ proper. If $M$ admits a collective completely integrable system then it must be multiplicity free.
\end{proposition}

\begin{proof} This theorem is proven in \cite[p. 223]{gs3} under a cleanness assumption: the image $W = \Phi(M)$ is a submanifold of $\mathfrak{g}^*$ and the map $\Phi\colon M \rightarrow W$ is a submersion \cite[p. 221]{gs3}. 

If $\Phi$ is proper then by the non-abelian convexity theorem for proper momentum maps \cite[Theorem 1.1]{lmtw}, $\square = \Phi(M) \cap \mathfrak{t}_+^*$ is convex. By Proposition \ref{principal stratum}, the $\square^{\text{rel-int}}$ is contained in the principal stratum $\sigma$ corresponding to $M$, so the set $G\cdot \square^{\text{rel-int}}$  is a submanifold of $\mathfrak{g}^*$. It follows from the symplectic cross-section theorem that $\Phi^{-1}(G\cdot \square^{\text{rel-int}})$ is an open dense subset of $M$ and the restricted map $\Phi\colon \Phi^{-1}(G\cdot \square^{\text{rel-int}}) \rightarrow G\cdot \square^{\text{rel-int}}$ is a submersion.  Thus $(\Phi^{-1}(G\cdot \square^{\text{rel-int}}),\omega,\Phi)$ satisfies the cleanness assumption.

Applying \cite[p. 223]{gs3}, it follows that $\Phi^{-1}(G\cdot \square^{\text{rel-int}})$ is multiplicity free so by continuity,  $M$ is multiplicity free. 
\end{proof}

\begin{proposition}\label{smooth fibres}
	Let $(M,\omega,\Phi)$ be a connected Hamiltonian $U(n)$ or $SO(n)$-manifold with $\Phi$ proper and let $F$ be a map constructed by Thimm's trick from a chain of subalgebras \eqref{gelfand-zeitlin chains}. 
	If the torus action generated by $F$ on the open dense subset $\mathcal{U}$ is completely integrable, then the fibres of $F$ are connected, embedded submanifolds.
\end{proposition}

\begin{remark}
	If a Hamiltonian torus action is not integrable, then it is possible that its fibres are not smooth. For example, the action of $S^1$ on $\mathbb{C}^2$ defined by $e^{i\theta}\cdot\left(z,w\right) = \left(e^{i\theta}z,e^{-i\theta}w\right)$ is generated by the momentum map $\mu(z,w) = \pi|z|^2-\pi|w|^2$, and the fiber $\mu^{-1}(0)$ is not an embedded submanifold of $\mathbb{C}^2$ (it is a cone).
\end{remark}
	
\begin{proof} Let $G = U(n)$ or $SO(n)$ and let $(M,\omega,\Phi)$ be a connected Hamiltonian $G$-manifold. Since $\Phi$ is proper and $\mathfrak{h}_n = \mathfrak{g}$, the map $F$ is proper. Thus by Theorem \ref{main theorem}, the fibres of $F$ are connected. 

Let
	\begin{equation}\label{gs map}
		F = \left( s\circ p_n^1\circ \Phi , \ldots, s\circ p_n^{n-1}\circ\Phi,s\circ\Phi\right) \colon M \longrightarrow \mathfrak{t}_1^* \oplus  \cdots \oplus \mathfrak{t}_{n}^*,
	\end{equation}
	be the map constructed by Thimm's trick with Guillemin and Sternberg's action coordinates from the chain \eqref{gelfand-zeitlin chains} and fix $(\xi_1, \ldots ,\xi_n) \in F(M)$. Let $\mathcal{O}_{\xi_n}$ be the $G$ coadjoint orbit through $\xi_n$. The fibre $ F^{-1}(\xi_1,\ldots , \xi_n)$ equals the preimage under $\Phi$ of the fiber $H^{-1}(\xi_1,\ldots , \xi_{n-1})$ the map
	\begin{equation}
		H = (s \circ p_{n}^1, \ldots , s \circ p_{n}^{n-1})\colon \mathcal{O}_{\xi_n} \rightarrow \mathfrak{t}_{1}^* \oplus \cdots \oplus \mathfrak{t}_{n-1}^*
	\end{equation}
	(which is a classical Gelfand-Zeitlin system).
	
	By Proposition \ref{gs completely integrable}, $M$ is multiplicity free. By Proposition \ref{woodward characterization 1}, $G$ acts transitively on $\Phi^{-1}(\mathcal{O}_{\xi_n})$, so it is an embedded submanifold. Since $\Phi$ is $G$-equivariant, the restricted map $\Phi\colon \Phi^{-1}(\mathcal{O}_{\xi_n}) \rightarrow \mathcal{O}_{\xi_n}$ is a submersion. Thus if $H^{-1}(\xi_1,\ldots , \xi_{n-1})$ is an embedded submanifold of $\mathcal{O}_{\xi_n}$, it follows that $ F^{-1}(\xi_1,\ldots , \xi_n)$ is an embedded submanifold of $M$.

It remains to show that the preimages $H^{-1}(\xi_1,\ldots , \xi_{n-1}) \subseteq \mathcal{O}_{\xi_n}$ are embedded submanifolds. Since every $U(n)$ coadjoint orbit is a multiplicity free $U(n-1)$-manifold\footnote{This was shown for generic $U(n)$ coadjoint orbits in \cite{gs3}. For arbitrary $U(n)$ and $SO(n)$ coadjoint orbits, this follows from Proposition \ref{gs completely integrable} and the fact that the classical Gelfand-Zeitlin systems are completely integrable (one can also give a direct proof using Proposition \ref{woodward characterization 1}). This in turn follows from the fact that for any coadjoint orbit the $T'$ action of the classical Gelfand-Zeitlin system is effective \cite[Proposition 4.2.2.]{pabiniak}  and the dimension of $T'$ is half the dimension of $\mathcal{O}$ \cite[Proposition 4.3.7.]{pabiniak}.} (respectively, every $SO(n)$ coadjoint orbit is a multiplicity free $SO(n-1)$-manifold), this follows inductively by applying the argument above. 
\end{proof}

\begin{remark}
	The geometry of the classical Gelfand-Zeitlin systems on $U(n)$ coadjoint orbits near these fibres have been studied by Eva Miranda and N.T. Zung \cite{miranda-zung} although no results have been published. A low dimensional example appears in \cite{alamiddine}. The fibres of bending flow systems on moduli spaces of oriented polygons (see Example \ref{bending flow systems}) were studied recently by Damien Bouloc, who showed that they are embedded coisotropic submanifolds and gave an explicit description of their geometry \cite{bouloc}. 
\end{remark}

\section{Examples}\label{s: examples}

In this section we give several examples (mainly in low dimensions) to illustrate the construction and its applications.

\begin{example}[Compact multiplicity free $SU(2)$ 4-manifolds]
	Let $(M^4,\omega)$ be a compact, connected, multiplicity free Hamiltonian $SU(2)$ 4-manifold with momentum map $\Phi$. These spaces are classified by their momentum set and principal isotropy group  \cite{iglesias}\footnote{These results are outlined in Section IV.5 of \cite{audin}.}.  
	
	Fix the maximal torus and positive Weyl chamber 
	\begin{equation}
		\begin{split}
			T &= \left\{ \left(\begin{array}{cc}
			e^{ix} & 0 \\ 0 & e^{-ix} 
		\end{array}\right) \colon x \in \mathbb{R} \right\},\\
		\mathfrak{t} &= \left\{\left(\begin{array}{cc}
			ix & 0 \\ 0 & -ix 
		\end{array}\right) \colon x \in \mathbb{R} \right\} \cong \mathbb{R},\, 
		 \mathfrak{t}_+ = \left\{\left(\begin{array}{cc}
			ix & 0 \\ 0 & -ix 
		\end{array}\right) \colon x \geq 0 \right\}\cong \mathbb{R}^+
		\end{split}
	\end{equation}
	and identify $\mathfrak{t} \cong \mathfrak{t}^*$ and $\mathfrak{t}_+ \cong \mathfrak{t}^*_+$ via the nondegenerate form $(X,Y) = \text{tr}(XY)$.  We can construct a Gelfand-Zeitlin system on $M$ from the chain $\mathfrak{t} \subseteq \mathfrak{su}(2)$,
	\begin{equation}
		F = \left(p\circ\Phi,s\circ \Phi\right)\colon M \rightarrow \mathbb{R} \times \mathbb{R}^+ \subseteq \mathbb{R}^2
	\end{equation}
	where $p\colon\mathfrak{su}(2)^* \rightarrow \mathfrak{t}^*$ is the projection and $s\colon \mathfrak{su}(2)^* \rightarrow \mathfrak{t}_+^*$ is the sweeping map.
	The momentum set of $M$ is an interval $s\circ\Phi(M) = [a,b] \subseteq \mathbb{R}^+$. By Proposition \ref{branching system} and the interlacing inequalities, the image of the Gelfand-Zeitlin system on $M$ is the set
	\begin{equation} 
		F(M) = \left\{ (x,y) \in \mathbb{R}^2 \colon \, a \leq y \leq b\text{ and } -y \leq x \leq y\right\}.
	\end{equation}
	If $a = 0$ then the principal isotropy group is trivial and $M$ is isomorphic to $\mathbb{C}P^2$ with the action of $SU(2)$ as a subgroup of $SU(3)$ and $\Phi^{-1}(0)$ is an isolated fixed point for the $SU(2)$ action \cite{iglesias}. One can check that the map $F$ is smooth on $M$ and generates an effective, completely integrable torus action whose weight lattice is \begin{equation}
	 	L^* = \mathbb{Z} \left\langle \left( 0,\frac{1}{\pi}\right),\left( \frac{1}{2\pi},\frac{1}{2\pi} \right) \right\rangle.
	 \end{equation}  
	With respect to the lattice structure, $F(M)$ is equivalent to the standard Delzant triangle of $\mathbb{C}P^2$.
	
	 If $a>0$ then $M$ is a Hirzebruch surface, the principal isotropy group is $\mathbb{Z}_m$ for some positive integer $m$, and $M$ is symplectomorphic to a blow-up of $\mathbb{C}P^2$ if $m$ is odd or $S^2 \times S^2$ if $m$ is even \cite{iglesias}. The map $F$ is smooth on $M$ and one checks that it generates an effective, completely integrable torus action on $M$ whose weight lattice is 
	 \begin{equation}
	 	L_m^* = \mathbb{Z} \left\langle \left( 0,\frac{m}{\pi}\right),\left( \frac{1}{2\pi},\frac{m}{2\pi} \right) \right\rangle
	 \end{equation}
	 if $m$ is odd and 
	 \begin{equation}
	 	L_m^* = \mathbb{Z} \left\langle \left( 0,\frac{m}{2\pi}\right),\left( \frac{1}{\pi},0 \right) \right\rangle.
	 \end{equation}
	 if $m$ is even.
	 Thus one sees that for all $m$, $F(M)$ is equivalent to a standard Delzant polytope of a Hirzebruch surface.
	 
\end{example}

Recall that following \cite{g}, the Gromov width of a symplectic manifold of dimension $2n$ is defined as 
\begin{equation}
	\text{GWidth}(M,\omega) = \sup_{r>0} \left\{ \pi r^2 \colon \exists \text{ a symplectic embedding } B^{2n}(r) \rightarrow M \right\}
\end{equation}
where $B^{2n}(r)$ is the open ball of radius $r>0$ in $\mathbb{R}^{2n}$ with the  standard symplectic structure.  

\begin{example}[$G_2$ coadjoint orbits]  There are four families of $G_2$ coadjoint orbits: the trivial orbit, the regular orbits, and the two one-parameter families of non-regular coadjoint orbits corresponding to maximal parabolic subgroups of $G_2$. Both of the non-regular coadjoint orbits admit Gelfand-Zeitlin systems; one can be viewed as a multiplicity free Hamiltonian $SU(3)$-manifold, for the action of the subgroup $SU(3) \leq G_2$ \cite[Example 7.4]{woodward1}, and the other is isomorphic to a $SO(8)$ coadjoint orbit. Proposition \ref{branching system} was applied in \cite{lane} to prove strict lower bounds for Gromov width of the former coadjoint orbit. The lower bound for the Gromov width of the latter coadjoint orbit follows from \cite{pabiniak1}. 
\end{example}

Another classical example from the integrable systems literature is the construction of integrable systems on cotangent bundles of homogeneous spaces.

\begin{example}[Cotangent bundles]
	Let $G$ act on a manifold $Q$ and consider the cotangent bundle $M = T^*Q$ with its canonical symplectic structure.  The cotangent lift of the action of $G$ is Hamiltonian.
	If the momentum map for this action is proper then we can apply Theorem \ref{main theorem} to a map $F$ constructed by Thimm's trick from a list of subalgebras that contains $\mathfrak{g}$.  In particular, if $Q=G/K$ is a homogeneous $G$-manifold, then $\Phi$ is proper and $M$ is a multiplicity free $G$-manifold  if and only if $(G,K)$ is a Gelfand-pair \cite{gs-mult-free}. If $G$ is one of $U(n)$ or $SO(n)$, then one can construct a Gelfand-Zeitlin system on $M$  \cite{gs3}.  
\end{example}

One can construct many interesting examples of non-compact multiplicity free $U(n)$ or $SO(n)$-manifolds by applying non-abelian symplectic cutting to cotangent bundles, as the next example shows.

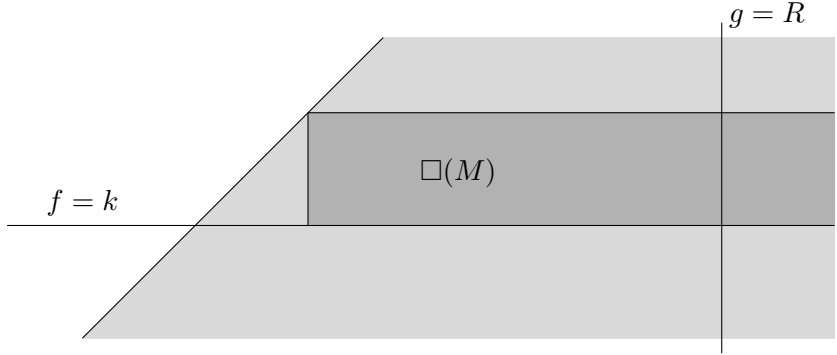
\begin{figure*}[t!]
   \centering
       \begin{tikzpicture}
  \draw [fill=gray!30!white, draw=gray!30!white, line width=0mm] (-2,0) -- (2,4) -- (8,4) -- (8,0) -- (-2,0);
  
  	\draw [fill=gray!60!white, draw=gray!60!white, line width=0mm] (1,3) -- (1,1.5) -- (8,1.5) -- (8,3) -- (1,3);
  	
  	\draw (-3,1.5) -- (8,1.5);
  	
  	\draw (1,3) -- (1,1.5);
  	
  	\draw (-2,0) -- (2,4);
  	
  	\draw (1,3) -- (8,3);
  	
  	\draw (6.5,4.2) -- (6.5,-0.2);
  	
  	\draw (-2,1.8) node{$f=k$};
  	
  	\draw (7.1,4.3) node{$g=R$};
  	
  	\draw (3,2.2) node{$\square(M)$};
	
  \end{tikzpicture}
    \caption{Momentum set of a noncompact multiplicity free $U(2)$-manifold. }
    \label{figure} 
\end{figure*}

\begin{example}[A non-compact $U(2)$-manifold with finite Gromov width] Consider the standard action of $U(2)$ on the unit sphere $S^3 \subseteq \mathbb{C}^2$. The cotangent lift of this action is a Hamiltonian action of $U(2)$ on $T^*S^3$. The momentum map for this action is proper and by Proposition \ref{woodward characterization 1}(2) the action is multiplicity free. Fix the maximal torus	
	\begin{equation}
		T = \left\{ \left( \begin{array}{cc}
			e^{ix_1} & 0 \\
			0 & e^{ix_2}
		\end{array}\right) \colon \, x_1,x_2 \in \mathbb{R}\right\}
	\end{equation}
	and identifications $\mathfrak{t}^* \cong \mathfrak{t} \cong \mathbb{R}^2$, 
	$\mathfrak{t}_+^* \cong \{ (x_1,x_2) \in \mathbb{R}^2 \colon \, x_1 \geq x_2\}$.  The momentum set of $T^*S^3$ is 
	\begin{equation}
		\square(T^*S^3) = \Phi(T^*S^3)\cap \mathfrak{t}_+^* = \left\{ (x_1,x_2) \in \mathbb{R}^2 \colon\,  0 \leq x_1 \text{ and } x_2 \leq 0 \right\}.
	\end{equation}
	Consider the functions
	\begin{equation}
		f,g\colon \mathfrak{t}_+^* \rightarrow \mathbb{R},\, f(x_1,x_2) = -x_2, \, g(x_1,x_2) = x_1.
	\end{equation}
	We can perform non-abelian symplectic cutting \cite{woodward1} with respect to the collective function $f\circ s$ at some level $k >0$ to produce a non-compact, multiplicity free $U(2)$-manifold $(M,\omega)$ with momentum set
	\begin{equation}
		\square(M) = \left\{ (x_1,x_2) \in \mathbb{R}^2 \colon \, 0 \leq x_1 \text{ and } -k \leq x_2 \leq 0  \right\}
	\end{equation}
	(see Figure \ref{figure}). Identify $U(1)$ with the subgroup of diagonal matrices $\text{diag}(e^{i\theta},1)$. The chain of subgroups $U(1) \leq U(2)$ equips $M$ with a Gelfand-Zeitlin system 
	\begin{equation}
		F = (p\circ\Phi,s\circ\Phi) \colon M \rightarrow \mathbb{R} \times \mathbb{R}^2.
	\end{equation}
	By Proposition \ref{branching system}, the image is
	\begin{equation}
		F(M) = \left\{ (x_0,x_1, x_2,) \in \mathbb{R}^3 \colon \,0 \leq x_1,\: -k \leq x_2 \leq 0,\text{ and } x_2 \leq x_0 \leq x_1 \right\}. 
	\end{equation}
	Since $M$ is multiplicity free, the map $F$ generates a completely integrable $T^3$-action on the open dense set $\mathcal{U} = M \setminus F^{-1}(0,0,0)$ (Theorem \ref{gs completely integrable}) and one can check that this action is effective. Combining this with Theorem \ref{main theorem}, it follows that $\left(\mathcal{U},\omega\vert_{\mathcal{U}},F\vert_{\mathcal{U}}\right)$ is a proper toric $T^3$-manifold. Note that with the identifications above, the weight lattice for $T^3$ is $\frac{1}{2\pi} \mathbb{Z}^3 \subseteq \mathbb{R}^3$. 
	
	Let $S\subseteq F(\mathcal{U})$ be the interior of the simplex with vertices $(0,0,0)$, $(0,0,-k)$, $(0,k,-k)$, and $(-k,0,k)$.  Since $\left(\mathcal{U},\omega\vert_{\mathcal{U}},F\vert_{\mathcal{U}}\right)$ is a proper toric $T^3$-manifold, the preimage $F^{-1}(S)$ inherits the structure of a proper toric $T^3$-manifold (in particular, by the convexity theorem for proper Hamiltonian torus manifolds \cite[Theorem 30]{kb}, the submanifold $F^{-1}(S)$ is connected).  By the classification of proper toric manifolds, $F^{-1}(S)$ is isomorphic to the proper toric $T^3$-manifold
	$$(S\times T^3,  \omega_0 =  dx_1\wedge dy_1 + dx_2\wedge dy_2 +dx_3\wedge dy_3, \text{pr}_1).$$
	By \cite[Proposition 2.1]{pabiniak1}, $2\pi k \leq  \text{GWidth}(S\times T^3, \omega_0)$ (in fact, this is an equality). It follows that 
	$$2\pi k \leq \text{GWidth}(F^{-1}(S),\omega\vert_{F^{-1}(S)}) \leq \text{GWidth}(M,\omega).$$

	On the other hand, given a symplectic embedding of a ball $B^6(r)$ into $M$, for every $\rho < r$ there is a $R>0$ such that the image of the closed  ball $\overline{B}^6(\rho)$ is contained in the sublevel set $g\circ s \circ \Phi(x)<R$ (see Figure \ref{figure}). 
	Performing non-abelian symplectic cutting by the collective function $g\circ s$ at the level $R$, we obtain a multiplicity free Hamiltonian $U(2)$-manifold $M_R$ with momentum set
	\begin{equation}
		\square(M_R) = \left\{ (x_1,x_2) \in \mathbb{R}^2 \colon \, 0\leq  x_1 \leq R \text{ and } -k \leq x_2 \leq 0  \right\}.
	\end{equation}
	By the classification of convex multiplicity free manifolds \cite[Theorem 11.2]{knop1} (and since the principal isotropy group for the action of $U(2)$ on $M_R$ is trivial), $M_R$ is isomorphic as a multiplicity free Hamiltonian $U(2)$-manifold to the $U(3)$-coadjoint orbit $(\mathcal{O}_{\lambda},\omega_{\lambda})$ with the same momentum set (this coadjoint orbit can be identified with the set of Hermitian $3\times 3$ matrices whose eigenvalues are $R,0,$ and $-k$).  Thus 
	$$\pi r^2 \leq \text{GWidth}(M_R,\omega) = \text{GWidth}(\mathcal{O}_{\lambda},\omega_{\lambda}) = \min\left\{ 2\pi k, 2\pi R \right\}\leq 2\pi k$$ 
	where the last equality follows from the known upper bound for Gromov width of coadjoint orbits of compact Lie groups \cite{castro}. Thus $\text{GWidth}(M,\omega) \leq 2\pi k$.  
	
	Therefore $\text{GWidth}(M,\omega) = 2\pi k$. In contrast, the Hofer-Zehnder capacity of $M$ is infinite: one can construct a sequence of admissible collective functions on $M$ with unbounded oscillation (see \cite{hofer-zehnder} for definition of the Hofer-Zehnder capacity).
	
\end{example}

If a compact symplectic manifold admits a completely integrable Hamiltonian torus action, then it also admits an invariant K\"ahler structure \cite{delzant}. As the following example demonstrates, this is not true for Gelfand-Zeitlin systems.  

\begin{example}[A G-Z system with no invariant K\"ahler structure] In \cite{tolman}, \newline Tolman constructed a symplectic 6-manifold with a Hamiltonian $T^2$-action that has no invariant K\"ahler structure. Woodward showed that such examples can be obtained as non-abelian symplectic cuttings $M_{\leq a}$ of $U(3)$ coadjoint orbits, considered as Hamiltonian $U(2)$-manifolds; there is no invariant K\"ahler structure for the action of the maximal torus of $U(2)$ on $M_{\leq a}$ \cite[Figure 3]{woodward}.  Being multiplicity free $U(2)$-manifolds, one can use the chain $U(1) \leq U(2)$ (as in the previous example) to construct a Gelfand-Zeitlin system on $M_{\leq a}$. This system generates an effective Hamiltonian $T^3$-action on an open dense subset of $M_{\leq a}$. The maximal torus of $U(2)$ acts on this open dense set as a subtorus of $T^3$, therefore there is no K\"ahler structure on $M_{\leq a}$ that is invariant under the $T^3$-action.

It was observed in \cite[Remark 3.9]{nnu} that the standard complex structure on a $U(n)$ coadjoint orbit is not invariant under the Gelfand-Zeitlin torus action. 
\end{example}

Our last example demonstrates an application of Proposition \ref{reduction 2} and a construction of a map by Thimm's trick that does not use a chain of subalgebras.

\begin{example}[Bending flow systems on moduli spaces]\label{bending flow systems} Let $H$ be a compact, connected Lie group and suppose that $M_1, \ldots, M_n$ are Hamiltonian $H$-manifolds. The symplectic direct sum $M = M_1 \times \cdots \times M_n$ is a Hamiltonian $G = H \times \cdots \times H$-manifold. Let $\mathcal{P}$ be a convex $n$-gon with vertices labelled $1,\ldots, n$ clockwise. Fix a triangulation $\mathcal{T}$ of $\mathcal{P}$.  To the interior edges of $\mathcal{T}$, which we denote $(i,j)$, $i<j$, we can associate subalgebras of the form 
\begin{equation}
  \mathfrak{h}_{i,j} = \left\{ (0,\ldots , 0,X,\ldots ,X,0,\ldots, 0) \in \mathfrak{g} \colon X \in \mathfrak{h} \right\} \cong \mathfrak{h}
\end{equation} 
where the $n$-tuple $(0,\ldots , 0,X,\ldots ,X,0,\ldots, 0)$ is zero in the first $i-1$ entries, equal to $X\in \mathfrak{h}$ in entries $i$ to $j$, and equal to zero in entries $j+1$ to $n$. If $i_1 \leq i_2 < j_2 \leq j_1$, then 
\begin{equation}
  [\mathfrak{h}_{i_1,j_1},\mathfrak{h}_{i_2,j_2}] \subseteq \mathfrak{h}_{i_2,j_2}.
\end{equation}
If $j_1<i_2$, then 
\begin{equation}
  [\mathfrak{h}_{i_1,j_1},\mathfrak{h}_{i_2,j_2}] = \{0\}\subseteq \mathfrak{h}_{i_2,j_2}.
\end{equation}
Thus for any triangulation $\mathcal{T}$, the subalgebras $\mathfrak{h}_{i,j}$, $(i,j) \in \mathcal{T}$ pairwise satisfy condition \ref{condition of thimm}. Further, if $\mathfrak{h}_{1,n}$ is the diagonal subalgebra, then $[\mathfrak{h}_{1,n}, \mathfrak{h}_{i,j}] \subseteq \mathfrak{h}_{i,j}$ for all $i< j$. Assuming properness, e.g. if $M$ is compact, we can apply Proposition \ref{reduction 2} to show that the map $\tilde F$ induced on the diagonally reduced space $M \sslash_0 H$ has convex image and connected fibres. 

For example, if $M_k = S^2_{r_k}$ is the sphere of radius $r_k>0$ and $H = SO(3)$, then this construction recovers the bending flow system on the moduli space of oriented polygons in $\mathbb{R}^3$, which can also be obtained as a reduction of the standard Gelfand-Zeitlin systems on coadjoint orbits \cite{hk,km}.  The fact that the open dense subsets where the bending flows generate Hamiltonian torus actions are proper toric manifolds was used in \cite{mandini} to prove lower bounds on the Gromov width of these moduli spaces.  The fibres of these systems have been studied by \cite{bouloc}. 

\end{example}

\section{Thimm's trick and symplectic contraction}\label{s:symplectic contraction}

A connection between maps constructed by Thimm's trick with Guillemin and Sternberg's action coordinates and maps constructed by contraction/degeneration was recently explored in \cite{hmm}.  Given a Hamiltonian $G$-manifold $(M,\omega,\Phi)$, Hilgert-Martens-Manon define a \emph{symplectic contraction map} $\Phi_M$ which is a continuous, surjective, and proper map from $M$ to a singular space $M^{sc}$ called the \emph{symplectic contraction of} $M$. $M^{sc}$ is the symplectic analogue of the horospherical contraction/degeneration of a reductive group action due to Popov and Vinberg \cite{pop}.  $M^{sc}$ is equipped with an action of $G\times T$ that is Hamiltonian in the appropriate sense for these singular spaces. The symplectic contraction map mimics the time-1 flow of the gradient-Hamiltonian vector field of a horospherical contraction/degeneration. The space $M^{sc}$ comes equipped with a continuous map $\mu_{\mathbb{T}}\colon M^{sc} \rightarrow \mathfrak{t}^*$ -- which is a momentum map (in the appropriate sense) for the action of $T$ on $M^{sc}$ -- such that the following diagram commutes

	\begin{equation*}
		\begin{CD}
			M @>\Phi_M >> M^{sc}\\
			@V\Phi VV @VV\mu_{\mathbb{T}}V\\
			\mathfrak{g}^* @>s >> \mathfrak{t}^*.
		\end{CD}
	\end{equation*}

	By construction, the space $M^{sc}$ is stratified by smooth symplectic manifolds.  If $\sigma \subseteq \mathfrak{t}_+^*$ is the principal stratum corresponding to $M$, then $\mu_{\mathbb{T}}^{-1}(\sigma)$ contains the open dense stratum of $M^{sc}$ and the restriction of $\Phi_M$ to the intersection of  $\Phi^{-1}(\Sigma_{\sigma})$ with the principal orbit-type stratum of the action of $G$ on $M$ is a symplectomorphism onto this stratum \cite[Proposition 4.3]{hmm}.

Given a chain of subgroups $H_1 \leq \cdots \leq H_d \leq G$ (or more generally, a chain of homomorphisms), Hilgert-Martens-Manon also define a \emph{branching contraction map} $\Phi_M$ and \emph{a branching contraction space} $M^{sc}$ with similar properties as before: the map $\Phi_M$ is continuous, surjective, and proper and there is an open dense submanifold on which the restriction of  $\Phi_M$ is a symplectomorphism \cite[Proposition 7.16]{hmm}. Further, there is a continuous map $\mu_{\mathbb{T}}\colon M^{sc}\rightarrow \mathfrak{t}_1^* \oplus \cdots \oplus \mathfrak{t}_d^*$ such that the following diagram commutes
	\begin{equation}\label{thimm:branching}
		\begin{CD}
			M @>\Phi_M >> M^{sc}\\
			@V \Phi VV @VV\mu_{\mathbb{T}} V\\
			\mathfrak{g}^* @>\left(s\circ p_1, \ldots ,s\circ p_d\right) >> \mathfrak{t}^*_{1}\oplus \cdots \oplus \mathfrak{t}_{d}^* 
		\end{CD}
	\end{equation}
	where the composition of $\Phi$ with $\left(s\circ p_1, \ldots ,s\circ p_d\right)$ is precisely the map $F$ constructed by Thimm's trick with Guillemin and Sternberg's action coordinates from the chain of subalgebras $\mathfrak{h}_1\subseteq \cdots \subseteq \mathfrak{h}_d$ \eqref{continuous map}.
	
	In this case, if for each $1\leq k \leq d$, $\sigma_k$ is the principal stratum of $\mathfrak{t}_{k,+}^*$ corresponding to the induced action of $H_k$ on $M$, then $\mu_{\mathbb{T}}^{-1}(\sigma_1 \times \cdots \times \sigma_d)$ contains the open dense stratum of the branching contraction space $M^{sc}$ and the restriction of $\Phi_M$ to the intersection of $\mathcal{U} = F^{-1}(\sigma_1 \times \cdots \times \sigma_d)$ with the principal orbit-type stratum for the action of $G$ is a symplectomorphism onto this open dense stratum \cite[Proposition 7.16]{hmm}. Since the diagram \eqref{thimm:branching} commutes and the contraction map $\Phi_M$ is continuous and surjective, we note that convexity and fibre connectedness for the map $F$ (Theorem \ref{main theorem}) is equivalent to convexity and fibre connectedness for the map $\mu_{\mathbb{T}}$ if either of these maps are proper (The fact that $F(M)$ is convex if and only if $\mu_{\mathbb{T}}(M^{sc})$ is convex follows immediately by surjectivity of $\Phi_M$. The fact that the fibres of $F$ are connected iff the fibres of $\mu_{\mathbb{T}}$ are connected follows by Lemma \ref{marginal connectedness} and the fact that the restriction of $\Phi_M$ to the open dense intersection of $\mathcal{U}$ and the principal orbit type stratum is a homeomorphism onto its image).

\providecommand{\bysame}{\leavevmode\hbox to3em{\hrulefill}\thinspace}
\providecommand{\MR}{\relax\ifhmode\unskip\space\fi MR }
\providecommand{\MRhref}[2]{%
  \href{http://www.ams.org/mathscinet-getitem?mr=#1}{#2}
}
\providecommand{\href}[2]{#2}

\end{document}